\documentclass[a4paper,12pt]{article}

\usepackage{amsthm}
\usepackage{latexsym}
\usepackage{dsfont}
\usepackage{bbm}
\usepackage{amssymb}
\usepackage{amsmath}
\usepackage{graphicx}

\numberwithin{equation}{section}

\theoremstyle{plain}
\newtheorem{Theorem}{Theorem}[section]
\newtheorem{Lemma}[Theorem]{Lemma}

\theoremstyle{remark}
\newtheorem{Rem}[Theorem]{Remark}
\theoremstyle{definition}
\newtheorem{Exa}[Theorem]{Example}

\DeclareMathOperator{\N}{\mathbb{N}}
\DeclareMathOperator{\Z}{\mathbb{Z}}

\DeclareMathOperator{\R}{\mathbb{R}}

\DeclareMathOperator{\V}{\mathbb{V}}

\DeclareMathOperator{\Prob}{\mathbb{P}}
\DeclareMathOperator{\Var}{Var}

\DeclareMathOperator{\E}{\mathbb{E}}
\DeclareMathOperator{\F}{\mathcal{F}}

\DeclareMathOperator{\1}{\mathbbm{1}}
\DeclareMathOperator{\D}{\mathfrak{D}}

\newcommand{\mm}{\mathcal{M}}
\newcommand{\mn}{\mathbb{N}}

\title{Exponential Rate of Almost Sure Convergence of Intrinsic Martingales in Supercritical Branching Random Walks}
\author{Alexander Iksanov\footnote{This work was commenced while A. Iksanov was visiting M\"unster in January 2009. A.\, Iksanov thanks G.\,Alsmeyer, M.\,Meiners and
 Institut f\"ur Mathematische Statistik for invitation, hospitality and financial support.}   \\
                Faculty of Cybernetics  \\
                National T. Shevchenko University of Kiev      \\
                01033 Kiev, Ukraine
                \and
                Matthias Meiners    \\
                Institut f\"ur Mathematische Statistik  \\
                Einsteinstra\ss e 62    \\
                D-48149 M\"unster, Germany}

\begin{document}

\thispagestyle{empty}
\maketitle

\begin{abstract}
We provide sufficient conditions which ensure that the intrinsic
martingale in the supercritical branching random walk converges
exponentially fast to its limit.
The case of Galton-Watson processes is particularly included so
that our results can be seen as a generalization of a result given
in the classical treatise by Asmussen and Hering. As an auxiliary
tool, we prove ultimate versions of two results concerning the
exponential renewal measures which may be interesting on its own
and which correct, generalize and simplify some earlier works.
\vspace{0,1cm}


\hphantom{2000 Mathematics Subject Classification: }Secondary: 60K05, 60G42
\end{abstract}

\section{Introduction and main result} \label{sec:Intro_and_main_results}

The Galton-Watson process is the eldest and probably best
understood branching process in probability theory. There is a
vast literature on different aspects of these processes ranging
from simple distributional properties to highly non-trivial
results on convergence in function spaces. In particular, in
\cite[Section II.3-4]{AH1983} Asmussen and Hering investigate the
rate of the a.s.\ convergence of the normalized supercritical
Galton-Watson process to its limit and among other things point
out a criterion for the exponential rate of convergence
\cite[Theorem 4.1(i)]{AH1983}. The aim of the present paper is to
prove a counterpart of this last result for branching random
walks (BRW, in short) which form a generalization of the
Galton-Watson processes.

We proceed with a formal definition of the BRW. Consider an
individual, the ancestor, which we identify with the empty tuple
$\varnothing$, located at the origin of the real line at time
$n=0$. At time $n=1$ the ancestor produces a random number $J$ of
offspring which are placed at points of the real line according to
a random point process $\mm = \sum_{i=1}^J \delta_{X_i}$ on $\R$
(particularly, $J = \mm(\R)$). We enumerate the ancestor's
children by $1,2,\ldots,J$ (note that we do not exclude the case
that $J = \infty$ with positive probability). The offspring of the
ancestor form the first generation. The population further evolves
following the subsequently explained rules. An individual $u =
u_1\ldots u_n$
of the $n$th generation with position $S(u)$ on the real line
produces at time $n+1$ a random number $J(u)$ of offspring which
are placed at random locations on $\R$ given by the positions of
the random point process $\delta_{S(u)} * \mm(u) =
\sum_{i=1}^{J(u)} \delta_{S(u) + X_i(u)}$ where $\mm(u) =
\sum_{i=1}^{J(u)} \delta_{X_i(u)}$ denotes a copy of $\mm$ (and
$J(u) = \mm(u)(\R)$). The offspring of individual $u$ are
enumerated by $u1 = u_1 \ldots u_n 1$,$\ldots$, $uJ(u) = u_1
\ldots u_n J(u)$, the positions of offspring individuals are
denoted by $S(ui)$, $i = 1,\ldots,J(u)$. It remains to state that
$(\mm(u))_{u \in \V}$ is assumed to be a family of i.i.d.\ point
processes. Note that this assumption does not imply anything about
the dependence structure of the random variables
$X_1(u),\ldots,X_{J(u)}(u)$ for fixed $u$. The point process of
the positions of the $n$th generation individuals will be denoted
by $\mm_n$. The sequence of point processes $(\mm_n)_{n \in
\mn_0}$ is then called \emph{branching random walk}. Throughout
the article, we assume that $\E J>1$ (supercriticality) which
means that the population survives with positive probability.
Notice that provided $J<\infty$ a.s., the sequence of generation
sizes in the BRW forms a Galton-Watson process.

An important tool in the analysis of the BRW is the Laplace transform
of the intensity measure $\xi := \E \mm$ of $\mm$,
\begin{equation*}    \label{eq:m}
m: [0,\infty) \to [0,\infty],   \qquad
\theta \mapsto \int_{\R} e^{-\theta x} \, \xi(dx) ~=~ \E \int_{\R} e^{-\theta x} \,\mm(dx).
\end{equation*}
We define $\D(m) := \{\theta \geq 0: m(\theta) < \infty\}$ and as a standing
assumption, suppose the existence of some $\gamma>0$ such that
$m(\gamma) < \infty$ (equivalently, $\D(m) \not = \emptyset$).
Possibly after the transformation $X_{i} \mapsto \gamma X_{i} +
\log m(\gamma)$ it is no loss of generality to assume $\gamma=1$
and
\begin{equation*}    \label{eq:m(1)=1}
m(1)    ~=~ \E \int_{\R} e^{-x} \,\mm(dx)   ~=~ \E \sum_{i=1}^J e^{-X_i}    ~=~ 1.
\end{equation*}
Put $Y_u:=e^{-S(u)}$ and
\begin{equation*}    \label{Sigman}
\overline{\Sigma}_n ~:=~ \E \sum_{|u|=n} Y_u \, \delta_{S(u)},
\qquad n\in\mn,
\end{equation*}
where $\sum_{|u|=n}$ denotes the summation over the individuals of
the $n$th generation, and let $(S_n)_{n \in\mn_0}$ denote a
zero-delayed random walk with increment distribution
$\overline{\Sigma}_1$. We call $(S_n)_{n \in\mn_0}$ the
\emph{associated random walk} .  It is well-known (see
\textit{e.g.} \cite[Lemma 4.1]{BK1997}) that, for any measurable
$f: \R^{n+1} \to [0,\infty)$,
\begin{equation}    \label{eq:distribution_of_S_n}
\E f(S_0,\ldots,S_n)    ~=~ \E \sum_{|u|=n} Y_u
f(S(u|\,0),S(u|\,1),\ldots,S(u)),
\end{equation}
where for $u = u_1 \ldots u_n$ we write $u|\,k$ for the individual $u_1 \ldots u_k$, the
ancestor of $u$ residing in the $k$th generation. We note, in
passing, that
\begin{equation*}
\varphi(t) ~:=~ \E e^{-tS_1} ~=~ m(1+t),    \quad   t \geq 0.
\end{equation*}
Define
\begin{equation*}\label{eq:W_n}
W_n ~:=~    \int_{\R} e^{-x} \,\mm_n(dx) ~=~ \sum_{|u|=n} Y_u, \ \
n\in\mn_0,
\end{equation*}
and denote the distribution of $W_1$ by $F$. Let $\F_n$ be the
$\sigma$-field generated by the first $n$ generations,
\textit{i.e.}, $\F_n = \sigma(\mm(u):\, |u|<n)$ where $|u|<n$
means $u \in \N^k$ for some $k<n$.

It is well-known and easy to check that $(W_n, \F_n)_{n \in\mn_0}$
forms a non-negative martingale and thus converges a.s.\ to a
random variable $W$, say, with $\E W\leq 1$. This martingale, which
is called the {\it intrinsic martingale in the BRW}, is of
outstanding importance in the asymptotic analysis of the BRW (see
\textit{e.g.} \cite{Gat2000} and \cite{Ner1981}). In this article,
we give sufficient conditions for the following statement to hold:
for fixed $a>0$
\begin{equation}    \label{eq:a.s._convergence_of_exponential_series}
\sum_{n \geq 0} e^{an} (W-W_n) \qquad    \text{converges a.s.}
\end{equation}
Clearly, \eqref{eq:a.s._convergence_of_exponential_series} states
that $(W_n)_{n\in\mn_0}$ a.s.\ converges to $W$ exponentially
fast.

There are already (at least) two articles which explore the rate
of convergence of the intrinsic martingale in the BRW to its
limit. In \cite{AIPR2009} necessary and sufficient conditions were
found for the series in
\eqref{eq:a.s._convergence_of_exponential_series} to converge in
$L_p$, $p>1$. Sufficient conditions for the a.s.\ convergence of
the series
\begin{equation*}
\sum_{n \geq 0} f(n)(W-W_n),
\end{equation*}
where $f$ is a function regularly varying at $\infty$ with an
index larger than $-1$, were obtained in \cite{Iks2006}. The
results derived in both the paper at hand and \cite{Iks2006} form
a generalization of the results in \cite[Section II.4]{AH1983},
where the rate of the a.s.\ convergence of the normalized
supercritical Galton-Watson process to its limit was investigated.
We want to remark that the scheme of our proofs borrows heavily
from the ideas laid down in \cite[Section II.4]{AH1983} but the
technical details are much more involved. The source of
complication can be easily understood: given $\mathcal{F}_n$, the
$W_{n+1}$ in the setting of the Galton-Watson processes is just
the sum of finite number of i.\,i.\,d. random variables whereas
the $W_{n+1}$ in the setting of the BRW is a weighted sum of,
possibly infinite, number of i.\,i.\,d. random variables.

Before stating our main results, we need some more notation and
explanations. If $0<\inf_{1 \leq \theta \leq 2}
m^{1/\theta}(\theta)<1$, then there exists a $\vartheta_0 \in
(1,2]$ such that $m^{1/\vartheta_0}(\vartheta_0) = \inf_{1 \leq
\theta \leq 2} m^{1/\theta}(\theta)$. The derivative of the
function $\theta \mapsto m^{1/\theta}(\theta)$ is well-defined and
negative on $(1,\vartheta_0)$ and the left derivative is
well-defined and non-positive on $(1,\vartheta_0]$. From this we
conclude that the left derivative of $m$ (to be denoted by
$m^\prime$ in what follows) is well-defined and negative on
$(1,\vartheta_0]$, \textit{i.e.}, \label{pa}
\begin{equation*}    \label{eq:m'(vartheta)<0}
m^\prime(\vartheta_0)<0.
\end{equation*}

\begin{Theorem} \label{Theorem:exponential_convergence:sufficiency}
Let $a>0$ be given. Assume that
\begin{equation}    \label{eq:sufficient1}
e^a \, m^{1/r}(r)   \leq 1  \qquad  \text{for some } r \in (1,2)
\end{equation}
and define $\vartheta$ to be the minimal $r > 1$ such that $e^{ar}
m(r) = 1$. Assume further that
\begin{equation}
\label{eq:sufficient2}
\E W_1^{\vartheta} ~<~ \infty,
\end{equation}
and in case when $a=-\log \inf_{r\geq 1}\, m^{1/r}(r)$ (which
implies $\vartheta=\vartheta_0$) assume that
\begin{equation*}    \label{eq:sufficient_boundary}
-\frac{\log m(\vartheta_0)}{\vartheta_0}    ~<~
-\frac{m^\prime(\vartheta_0)}{m(\vartheta_0)}.
\end{equation*}
Then \eqref{eq:a.s._convergence_of_exponential_series} holds true.
\end{Theorem}

\begin{figure}[htb]
    \begin{center}
        \includegraphics[width=0.64\textwidth]{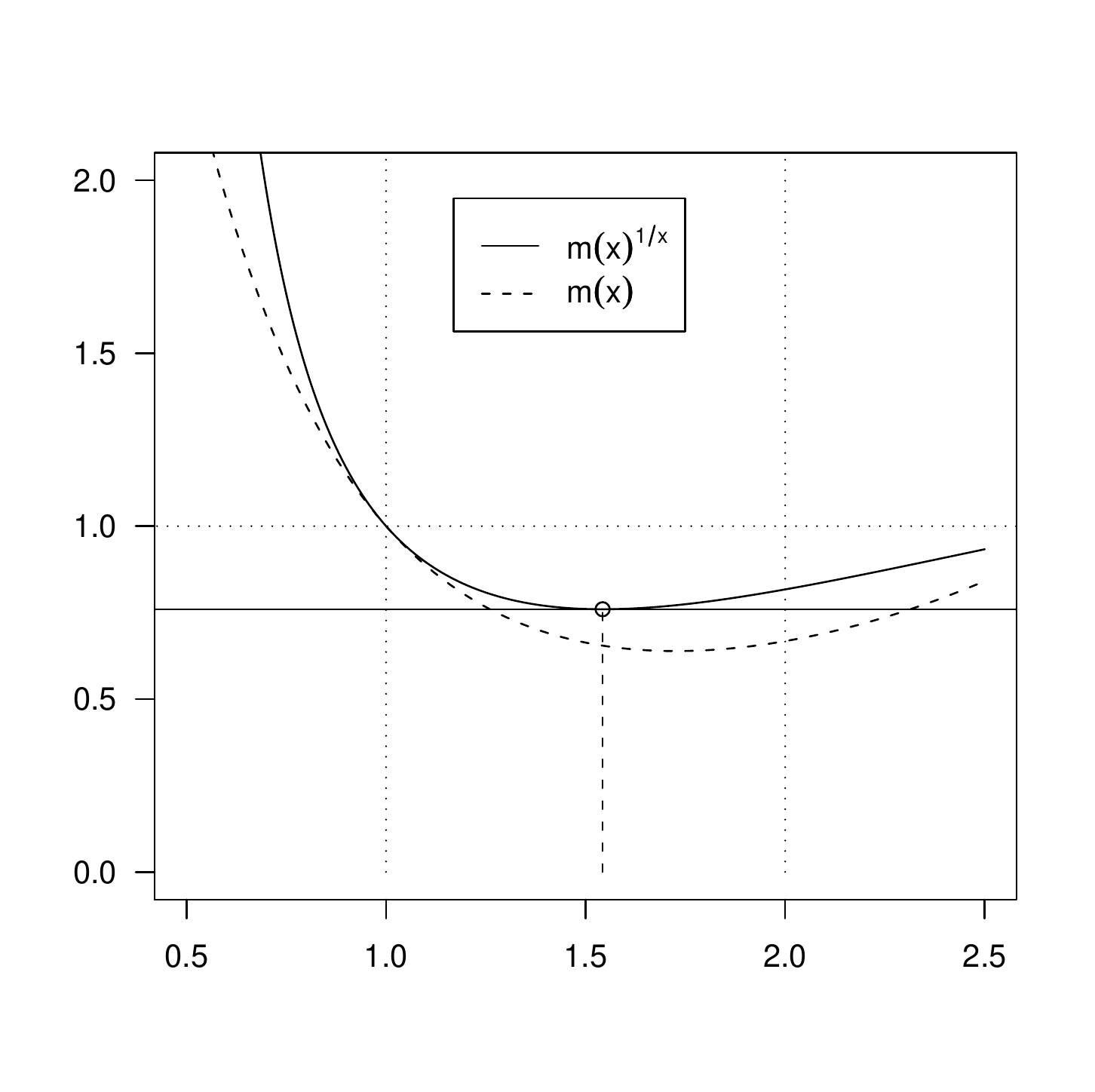}
    \end{center}
    \caption{A typical situation in which Theorem \ref{Theorem:exponential_convergence:sufficiency}
    applies. Here $m(1) = 1$, and $m$ is strictly decreasing in a right neighborhood of $1$. The bottom point
    of the  graph of $m^{1/x}(x)$ is marked by a circle. The vertical dashed line connects this point to
    the $x$-axis indicating the point $\vartheta_0$.
    The solid horizontal line and the dotted horizontal line at $1$ indicate the interval of possible values of
    $e^{-a}$ such that $a > 0$ and $e^a m^{1/\vartheta_0} m(\vartheta_0) < 1$.
    For those $a$'s, the assumptions of Theorem \ref{Theorem:exponential_convergence:sufficiency} are
    satisfied. The vertical dotted lines at $1$ and $2$ emphasize the importance of the interval $(1,2)$
    in which $\vartheta_0$ is supposed to be located.}
\label{fig:m_plot}
\end{figure}

\begin{Rem}
The point $(\vartheta_0, m^{1/\vartheta_0}(\vartheta_0))$ either
belongs to the strictly decreasing branch of the graph
$\{(x,m^{1/x}(x)): x \in \D(m)\}$, equivalently, $-\log
m(\vartheta_0)/\vartheta_0 < -
m^\prime(\vartheta_0)/m(\vartheta_0)$, or it is the bottom point
of that graph, which is equivalent to $\log
m(\vartheta_0)/\vartheta_0 =
m^\prime(\vartheta_0)/m(\vartheta_0)$. From this we conclude that
the theorem implies
\eqref{eq:a.s._convergence_of_exponential_series} with $a=-\log
m(\vartheta_0)/\vartheta_0$ to hold when the former occurs.
Intuitively, while the second situation is somewhat exceptional,
the first situation is more or less typical. In conclusion,
condition $e^a \inf_{1 \leq r \leq 2}\, m^{1/r}(r)<1$ is
``typically'' sufficient for
\eqref{eq:a.s._convergence_of_exponential_series} to hold. A
similar remark with an obvious modification also applies to
Theorem \ref{Theorem:Heyde1}(a) and Theorem \ref{Theorem:Heyde}.
\end{Rem}

\begin{Rem}
Let $p\in (1,2)$. By using a completely different argument, in
\cite{AIPR2009} it was proved that conditions
$$\E W_1^r<\infty \ \ \text{and} \ \ e^am^{1/r}(r)<1 \ \ \text{for some} \ \
r\in[p,2]$$ are sufficient for the $\mathcal{L}^p$, and hence, the
a.s.\, convergence of $\sum_{n\geq 0}e^{an}(W-W_n)$. Plainly, the
conditions of our Theorem
\ref{Theorem:exponential_convergence:sufficiency} are weaker.
\end{Rem}
\begin{Rem}
Under the assumptions of Theorem
\ref{Theorem:exponential_convergence:sufficiency}, the martingale
$(W_n)_{n \in\mn_0}$ is uniformly integrable, equivalently,
$\Prob\{W>0\}>0$. An ultimate criterion of uniform integrability
of the intrinsic martingale was recently presented in
\cite{AI2009}, following earlier investigation in \cite{Big1977a,
Liu, Lyo1997}.
\end{Rem}

\newpage
\begin{Exa}[Galton-Watson processes]    \label{Exa:GWP}
Suppose that $m:= \E J \in (1,\infty)$ and that $e^{-X_i} = m^{-1}
\1_{\{i \leq J\}}$, $i\in\N$. Then $(W_n)_{n\in\N_0}$ forms a
normalized supercritical Galton-Watson process. Pick $p$ and $q$
such that $p\in (1,2)$ and $1/p + 1/q = 1$. Theorem 4.1 in
\cite{AH1983} proves that
\begin{equation}    \label{eq:AH1}
W-W_n ~=~ o(m^{-n/q})   \qquad   \text{a.s.\ as } n \to \infty,
\end{equation}
if and only if
\begin{equation}    \label{eq:AH2}
\E W_1^p ~<~ \infty.
\end{equation}
Sufficiency of condition \eqref{eq:AH2} for \eqref{eq:AH1} to hold
follows from our Theorem
\ref{Theorem:exponential_convergence:sufficiency}. To see this,
take $a=q^{-1}\log m$ and notice that equality $m(\theta) =
m^{1-\theta}$, $\theta \geq 0$, implies that $e^{a} m^{1/r}(r) =
m^{1/r-1/p}< 1$ for $r>p$. Therefore, \eqref{eq:sufficient1} holds
(with strict inequality). Further, $\vartheta$ defined in Theorem
\ref{Theorem:exponential_convergence:sufficiency} equals $p \in (1,2)$
in the present situation, which shows that \eqref{eq:sufficient2} holds.
By Theorem \ref{Theorem:exponential_convergence:sufficiency},
\eqref{eq:AH1} holds.
\end{Exa}

The rest of the article is organized as follows. In Section
\ref{sec:auxiliary_results}, we present some auxiliary
renewal-theoretic results which correct, generalize and simplify
some early results from \cite{Hey1964, Hey1966}. Results of this
section are an important ingredient of the proof of Theorem
\ref{Theorem:exponential_convergence:sufficiency} which is given
in Section \ref{sec:sufficiency}.

\section{Ultimate results for the exponential renewal function} \label{sec:auxiliary_results}

For a random variable $T$ with proper distribution which we
assume to be non-degenerate at $0$ let $\psi$ be its Laplace
transform:
\begin{equation*}
\psi:[0,\infty) \to (0,\infty], \quad \psi(t) ~:=~ \E e^{-t T}.
\end{equation*}
In what follows, we denote by $\psi^\prime$ the \emph{left} derivative of $\psi$.

Set $R := - \log \inf_{t \geq 0} \psi(t)$. Then $R \geq 0$ since
$\psi(0) = 1$, and unless $T\geq0$ a.s., the infimum
in the definition of $R$ is attained, \textit{i.e.}, there exists
some \label{ga0} $\gamma_0 \in [0,\infty)$ such that
$\psi(\gamma_0) = e^{-R}$. Note that $\gamma_0=0$ is equivalent to
$R=0$ since we assume the distribution of $T$ to be non-degenerate
at $0$. When $R>0$ and $a\in (0,R]$, let $\gamma$ denote the
minimal (finite) $t>0$ satisfying $\psi(t) = e^{-a}$ if such a $t$ exists.
Notice that $\gamma=\gamma_0$ if $a=R$ and the infimum is attained.
Let $(T_n)_{n \in \N_0}$ be a
zero-delayed random walk with a step distributed like $T$.
Whenever $\gamma$ as above exists, we use it to define a new probability
measure $\Prob_{\gamma}$ such that
\begin{equation}    \label{eq:P_gamma}
\Prob_{\gamma}(T_n\in A)
~=~ \psi(\gamma)^{-n} \E e^{-\gamma T_n} \1_{\{T_n\in A\}},
\qquad   n \in \N_0,
\end{equation}
for any Borel set $A \subseteq \R$. As a consequence of $\E_{\gamma}
e^{\gamma T_1} = \psi(\gamma)^{-1} < \infty$, we have
\begin{equation}    \label{eq:E_gamma_T_1^+<infty}
\E_{\gamma} T_1^+   ~<~ \infty.
\end{equation}
Since $\psi$ is non-increasing on $[0,\gamma_0]$ we conclude that if $\gamma_0 > 0$, then
\begin{equation*}
\E_{\gamma_0} T_1 ~=~ \E e^{-\gamma_0 T}T ~=~ -\psi'(\gamma_0)
\end{equation*}
should be non-negative and finite in view of \eqref{eq:E_gamma_T_1^+<infty}.

The first theorem in this section investigates finiteness of the
\emph{exponential} renewal function
\begin{equation*}    \label{eq:V}
V(x)    ~:=~    \sum_{n \geq 0} e^{an} \Prob(T_n \leq x),   \qquad  x\in\R.
\end{equation*}
\begin{Theorem} \label{Theorem:Heyde1}
Assume that $\Prob(T=0) \not = 1$ and let $a>0$ be given.
\begin{itemize}
\item[(a)] Assume that $\Prob(T<0)>0$.
\begin{itemize}
\item[(i)] If $a \in (0,R)$, then $V(x)$ is finite for every
$x\in\R$. \item[(ii)] If $a=R>0$ and
\begin{equation}    \label{eq:psi'(gamma_0)<0}
-\psi^\prime(\gamma_0)  ~=~ \E e^{-\gamma_0 T} T    ~>~ 0,
\end{equation}
then $V(x)$ is finite for every $x\in\R$.
\item[(iii)]
If $a>R$, then $V(x)=+\infty$ for all $x\in\R$.
\item[(iv)]
If $a=R>0$ and $\psi^\prime(\gamma_0)=0$ (equivalently, if
\eqref{eq:psi'(gamma_0)<0} does not hold), then $V(x)=+\infty$ for all $x\in\R$.
\end{itemize}
\item[(b)]
Assume that $\Prob(T>0)=1$. Then $V(x)$ is finite for every $x \in
\R$.
\item[(c)]
Assume that $\Prob(T\geq 0)=1$ and $\beta:=\Prob(T=0)>0$. Then
$R=-\log \beta$ and if $a\in (0,R)$, then $V(x)$ is finite for
every $x\in\R$, and if $a\geq R$, then $V(x)$ is infinite for all
$x \geq 0$.
\end{itemize}
\end{Theorem}

Theorem \ref{Theorem:Heyde1} constitutes a generalization of
Theorem B in \cite{Hey1964} but can also be partly deduced
(excluding the case $a=R$) from the more general Theorem 2 in
\cite{Bor1962}. Our contribution here is a streamlined derivation
of the exact value of $R$, a simple proof of dichotomy $a<R$
versus $a>R$ and investigating the most delicate case $a=R$.

The main tool for the analysis in Section \ref{sec:sufficiency} is
the following result, which provides the asymptotic behavior of the
exponential renewal function $V(x)$. Note in advance that Theorem
\ref{Theorem:Heyde} will be applied to $(S_n-an)_{n\in\N_0}$ where
$(S_n)_{n\in\N_0}$ is the associated random walk of the given BRW.
\begin{Theorem} \label{Theorem:Heyde}
Let $a>0$ be given. Assume that either $a\in (0,R)$ or $a=R$ and
\eqref{eq:psi'(gamma_0)<0} holds. Then, with $\gamma$ being the
minimal $t> 0$ satisfying $\psi(t) = e^{-a}$,
\begin{equation}    \label{eq:asymptotics_of_exp_renewal_measure}
V(x)    ~\sim~  \frac{e^{-a}}{\gamma (-\psi'(\gamma))} \,
e^{\gamma x}, \qquad \text{as}   \quad x \to \infty
\end{equation}
if $(T_n)_{n \in \N_0}$ is a non-arithmetic random walk, and
\begin{equation}    \label{eq:asymptotics_of_exp_renewal_measure1}
V(\lambda n)    ~\sim~  \frac{\lambda
e^{-a}}{(1-e^{-\lambda\gamma})(-\psi^\prime(\gamma))} \, e^{\gamma
\lambda n}, \qquad \text{as\ }\qquad n \to \infty
\end{equation}
if $(T_n)_{n\in\N_0}$ is arithmetic with span $\lambda>0$.
Moreover, in the arithmetic case,
\begin{equation}    \label{eq:asymptotics_of_exp_renewal_measure2}
V(\lambda n)-V(\lambda(n-1))    ~\sim~  \frac{\lambda
e^{-a}}{(-\psi^\prime(\gamma))} \, e^{\gamma \lambda n}, \qquad
\text{as\ }\qquad n \to \infty.
\end{equation}
\end{Theorem}
\begin{Rem} (a) Theorem \ref{Theorem:Heyde} describes the asymptotics of
$V(x)$ whenever it is finite.\newline (b) Provided $a<R$ or $a=R$
and \eqref{eq:psi'(gamma_0)<0} holds, equation $\psi(t)=e^{-a}$
has positive solutions.
\end{Rem}

Theorem \ref{Theorem:Heyde} is a generalization and correction of
Theorem 4 in \cite{Hey1966}\footnote{We think that an error in the
proof of Theorem 4 of the afore-mentioned paper comes from the end
of p.\,706 in \cite{Hey1966} where the dependence of the real
number $\xi(n) \in (\beta,\mu)$ on $n$ cannot be ignored since
possibly $\xi(n) \to \mu$ and then $\Prob(S_n \leq \xi(n))$ does
not necessarily decay at an exponential rate (in this footnote we
retained the original notation from \cite{Hey1966}).} the
differences being that
\begin{itemize}
    \item we do
not assume that $\E |T|$ is finite;
    \item the exponential (wrong) rate
$a/\E T$ claimed in \cite{Hey1966} under the assumption $\E T\in
(0,\infty)$ is replaced by the rate $\gamma$ in the non-arithmetic
case, and a similar substitution is proved to hold true in the
arithmetic case;
    \item unlike \cite{Hey1966} we treat, among others,
the boundary case $a=R$.
\end{itemize}
\begin{proof}[Proof of Theorem \ref{Theorem:Heyde1}]{\sc Case} (a).
(i). If $R=0$, then condition $a \in (0,R)$ cannot hold. So assume
that $R>0$, $a\in (0,R)$ and pick any $x \in \R$. With $\gamma_0$
defined at the beginning of the section choose $r\in (0,\gamma_0)$
such that $a < - \log \psi(r)$. Now use Markov's inequality to
obtain
\begin{equation*}
\sum_{n \geq 0} e^{an} \Prob(T_n \leq x) ~\leq~ \sum_{n \geq 0}
e^{an} e^{rx} \E e^{-rT_n} ~=~ e^{rx} \sum_{n \geq 0} e^{n(a +
\log \psi(r))} ~<~ \infty,
\end{equation*}
which proves the assertion under the assumption (i).

(ii) and (iv). Assume that $a=R>0$. The function
$g(y):=e^{-\gamma_0 y}\1_{[0,\infty)}(y)$ is directly Riemann
integrable. If \eqref{eq:psi'(gamma_0)<0} holds (does not hold), then the random
walk $(T_n)_{n \in\N_0}$ is transient (recurrent) under
$\Prob_{\gamma_0}$, the probability measure defined in
\eqref{eq:P_gamma}. As a consequence, the renewal measure
$U_{\gamma_0}$ of $(T_n)_{n\in \N_0}$ under $\Prob_{\gamma_0}$
satisfies $U_{\gamma_0}(I) < \infty$ ($=\infty$) for any open
non-empty interval $I$ if $(T_n)$ is non-arithmetic, and for any
open non-empty interval $I$ which contains some point $n\lambda$,
$n\in\Z$, if $(T_n)$ is arithmetic with span $\lambda$, respectively. Therefore,
if \eqref{eq:psi'(gamma_0)<0} holds, then
\begin{eqnarray*}
\sum_{n\geq 0}e^{an}\Prob(T_n\leq x) & = & \sum_{n\geq 0}
\E_{\gamma_0}
e^{\gamma_0 T_n}\1_{\{T_n\leq x\}}\\
&=& e^{\gamma_0 x}\sum_{n\geq 0}\E_{\gamma_0} g(x-T_n) ~<~ \infty,
\qquad x\in\R.
\end{eqnarray*}
Whereas, if \eqref{eq:psi'(gamma_0)<0} does not hold, then
\begin{equation*}
\sum_{n \geq 0}e^{an}\Prob(T_n\leq x)
~=~ e^{\gamma_0 x} \sum_{n \geq 0} \E_{\gamma_0} g(x-T_n)
~=~ \infty, \qquad x \in \R.
\end{equation*}
(Notice that this argument with $\gamma_0$ replaced by $\gamma$ also
applies in the situation of (a)(i).)

(iii). To complete the proof of (a) it remains to check that
$V(x)=+\infty$ for all $x\in\R$ provided $a>R$. Notice that the
case $R=0$ is not excluded and is equivalent to $\gamma_0 = 0$.
Further notice that $\psi$ assumes its infimum on $[0,\infty)$
since we assume $\Prob(T<0)>0$. Recall that
the unique minimizer of $\psi$ is denoted by $\gamma_0$ and
that $\psi'(\gamma_0)$, the left derivative of $\psi$ at $0$,
exists and is $\leq 0$ if $\gamma_0 > 0$.
\newline {\sc Subcase} (iii-I): $\gamma_0>0$. If
$\psi'(\gamma_0) < 0$, then, for any $c > 0$, we consider a
zero-delayed random walk, $(T_{c,n})_{n \in\N_0}$ say, with steps
distributed like $T \1_{\{T \geq -c\}}$. Set $\psi_c(t):=\E
e^{-tT_{c,1}}$ and notice that $\psi_c$ is finite on $[0,\infty)$
and that $R_c:=-\log \inf_{t \geq 0} \psi_c(t) \geq R$. If $c$ is
large enough, $\psi_c(t) \to \infty$ as $t \to \infty$. Thus
$\psi_c$ has a unique minimizer on $[0,\infty)$, $\gamma_c$ say,
and $\psi_c'(\gamma_c) = 0$. It is easily seen that $\gamma_0 \leq
\gamma_c$ and that $\gamma_c \downarrow \gamma_0$ as $c \uparrow
\infty$. Some elementary analysis now shows that $R_c$ converges
to $R$ as $c \to \infty$. Moreover,
\begin{equation}    \label{eq:V_a_geq_V_ca}
\sum_{n \geq 0} e^{an} \Prob(T_n \leq x) ~\geq~
\sum_{n \geq 0} e^{an} \Prob(T_{c,n} \leq x), \qquad    x \in \R.
\end{equation}
Therefore, if we can prove that provided $a>R_c$ the series
$\sum_{n \geq 0} e^{an} \Prob(T_{c,n} \leq x)$ diverges, this will
imply (after choosing $c$ sufficiently large) that provided $a>R$
the series $\sum_{n \geq 0} e^{an} \Prob(T_n \leq x)$ diverges.
Thus we have shown that, without loss of generality, we can work
under the additional assumption $\psi'(\gamma_0) = 0$. Condition
$a>R$ now reads as $\psi(\gamma_0)>e^{-a}$. By using the
probability measure $\Prob_{\gamma_0}$ defined in
\eqref{eq:P_gamma} we conclude that $\E_{\gamma_0} T_1 = 0$.
Hence, the random walk $(T_n)_{n \in\N_0}$ is recurrent under
$\Prob_{\gamma_0}$. As a consequence, the renewal measure
$U_{\gamma_0}$ of $(T_n)_{n\in \N_0}$ under $\Prob_{\gamma_0}$
satisfies $U_{\gamma_0}(I) = \infty$ for any open non-empty
interval $I$ if $(T_n)$ is non-arithmetic, and for any open
non-empty interval $I$ which contains some point $n\lambda$,
$n\in\Z$, if $(T_n)$ is arithmetic with span $\lambda$,
respectively. As a consequence,
\begin{eqnarray*}
\sum_{n \geq 0} e^{an} \Prob(T_n \leq x) & = &
\sum_{n \geq 0} (\psi(\gamma_0)e^a)^n \E_{\gamma_0} e^{\gamma_0 T_n} \1_{\{T_n \leq x\}}    \\
& \geq & \int_{(-\infty, x]} e^{\gamma_0 y} \, U_{\gamma_0}(dy)
~=~ \infty
\end{eqnarray*}
for every $x\in\R$.\newline {\sc Subcase} (iii-II): $\gamma_0=0$. We
have $\psi(t)\in (1,\infty]$ for all $t> 0$. If $T_n \to -\infty$
a.s., then $\Prob(T_n \leq x) \to 1$, as $n \to \infty$, and the
infinite series $\sum_{n \geq 0}e^{an}\Prob(T_n\leq x)$ diverges.
Thus we are left with the situation that either $T_n \to \infty$
a.s.\ or $(T_n)_{n \in\mn_0}$ oscillates. In both cases, $\E
T_{1,c} \in (0,\infty]$, where $T_{1,c}$ is defined as above, and
the Laplace transform $\psi_c$ of $T_{c,1}$ is finite on
$[0,\infty)$ and assumes its minimum at some $\gamma_c > 0$
satisfying $\psi_c'(\gamma_c) = 0$. Now we can argue as in the
subcase (iii-I) to show that $\sum_{n \geq 0} e^{an} \Prob(T_n
\leq x) = \infty$ for all $x \in \R$.

{\sc Case} (b). In this case $R=+\infty$. Pick any $x>0$ ($V(x)=0$
for $x\leq 0$) and choose $r>0$ such that $a < - \log \psi(r)$ and
proceed in the same way as under (a)(i).

{\sc Case} (c). Note that we have $\beta \in (0,1)$. Choose $a \in
(0,-\log \beta)$. The subsequent proof literally repeats that
given under the assumption (a)(i).

\noindent
Conversely, $\Prob(T_n=0)=\beta^n$, $n\in\N_0$. Therefore, if
$a\geq -\log \beta$, then $V(0)=+\infty$ which implies that
$V(x)=+\infty$ for all $x\geq 0$.
\end{proof}

\begin{proof}[Proof of Theorem \ref{Theorem:Heyde}]
By Theorem \ref{Theorem:Heyde1}, $V(x)<\infty$ for every $x\in\R$.
Under $\Prob_{\gamma}$, the probability measure defined in
 \eqref{eq:P_gamma}, $(T_n)_{n \in \N_0}$ forms a random walk with Laplace
transform
\begin{equation}    \label{eq:phi_gamma}
\psi_{\gamma}(t)    ~=~ \E_{\gamma} e^{-t T_1}
~=~ e^a \E e^{-(\gamma+t) T}    ~=~ e^{a} \psi(\gamma+t)
\end{equation}
and drift
\begin{equation}    \label{eq:nu_gamma}
\nu_{\gamma}    ~=~ -\psi_{\gamma}'(0)  ~=~ -e^{a} \psi'(\gamma)    ~\in~   (0,\infty).
\end{equation}
For $x \in \R$, we write $V(x)$ in the following form
\begin{equation}    \label{eq:exp_renewal_measure_in_terms_of_U_theta}
V(x)
~=~ \sum_{n \geq 0} \E_{\gamma} e^{\gamma T_n} \1_{\{T_n \leq x\}}  \\
~=~ \int_{(-\infty,x]} \!\! e^{\gamma y} \, U_{\gamma}(dy)
~=:~ e^{\gamma x} Z(x),
\end{equation}
where $U_{\gamma}$ denotes the renewal measure of the process
$(T_n)_{n \geq 0}$ under $\Prob_{\gamma}$.

\noindent
Assume that $(T_n)_{n\in\N_0}$ is non-arithmetic. Since
\begin{eqnarray*}
Z(x)
& = &
e^{-\gamma x} \int_{(-\infty,x]} \!\! e^{\gamma y} \, U_{\gamma}(dy)    \\
& = & \int e^{-\gamma(x-y)} \1_{[0,\infty)}(x-y) \,U_{\gamma}(dy)
\end{eqnarray*}
and the function $x \mapsto e^{-\gamma x} \1_{[0,\infty)}(x)$ is
directly Riemann integrable we can invoke the key renewal theorem
on the whole line to conclude that
\begin{eqnarray*}
e^{-\gamma x} \sum_{n \geq 0} e^{an} \Prob(T_n \leq x) & = &
Z(x)    \\
& \underset{x \to \infty}{\longrightarrow} &
\frac{1}{\nu_{\gamma}} \int_0^{\infty} e^{-\gamma y} \, dy ~=~
\frac{1}{\gamma \nu_{\gamma}},
\end{eqnarray*}
where we have used $\nu_{\gamma} > 0$. This in combination with
\eqref{eq:nu_gamma} immediately implies
\eqref{eq:asymptotics_of_exp_renewal_measure}.

\noindent Asymptotics
\eqref{eq:asymptotics_of_exp_renewal_measure1} in the arithmetic
case can be treated similarly. Finally,
\eqref{eq:asymptotics_of_exp_renewal_measure2} follows by an
application of \eqref{eq:asymptotics_of_exp_renewal_measure1} to
$$e^{-\gamma\lambda n}\left(V(\lambda n)-V(\lambda(n-1))\right)=e^{-\gamma\lambda n}V(\lambda n)-e^{-\gamma\lambda }\left(e^{-\gamma\lambda(n-1)}V(\lambda(n-1))\right).$$
\end{proof}

\section{Proof of Theorem \ref{Theorem:exponential_convergence:sufficiency}}  \label{sec:sufficiency}

For any $u \in \V$, let $W_1(u)$ denote a copy of $W_1$ but based on the point process $\mm(u)$ instead of $\mm = \mm(\varnothing)$, that is, if $W_1 = \psi(\mm)$ for an appropriate measurable function $\psi$, then $W_1(u) := \psi(\mm(u))$. In this situation, let
\begin{eqnarray*}
\widetilde{W}_{n+1} & := &  \sum_{|u|=n} Y_u W_1(u) \1_{\{e^{an} Y_u W_1(u) \leq 1\}}   \\
\text{and}  \qquad
R_n &   := & \E(W_n - \widetilde{W}_{n+1} \,|\, \F_n), \qquad   n \in \N_0.
\end{eqnarray*}
For $a>0$ define a measure $V_a$ on by
\begin{equation*}
V_a(x) ~:=~ V_a((0,x]) ~:=~ \sum_{n \geq 0} e^{an} \Prob(S_n-an \leq \log x)    \qquad   (x>0).
\end{equation*}

Given next is a result on the asymptotic behavior of $V_a$ and two
integrals involving $V_a$ which play an important role in the
proof of Theorem \ref{Theorem:exponential_convergence:sufficiency}.

\begin{Lemma}   \label{Lem:asymptotics_of_V_a}
Let $a>0$ be given. Assume that
\begin{equation*}
e^a m^{1/r}(r) \leq 1   \qquad  \text{for some } r >1
\end{equation*}
and define $\vartheta$ to be the minimal $r > 1$ such that $e^a m^{1/r}(r) = 1$. In case when $a=-\log \inf_{r\geq 1}\, m^{1/r}(r)$
(which implies $\vartheta=\vartheta_0$) assume further that
\begin{equation}    \label{eq:critical_derivative_negative}
-\frac{\log m(\vartheta_0)}{\vartheta_0}
~<~ -\frac{m^\prime(\vartheta_0)}{m(\vartheta_0)}.
\end{equation}
Then, as $x \to \infty$,
\begin{eqnarray}    \label{eq:asymptotic_of_V_a}
V_a(x)
& \sim &
\frac{x^{\vartheta-1}}{(\vartheta-1)(e^{a\vartheta}(-m'(\vartheta))-a)}, \\
\label{eq:comp_asymptotics_of_V_a} \int_{(0,x]} \!\! y \, V_a(dy)
& \sim & \frac{x^{\vartheta}}{\vartheta
(e^{a\vartheta}(-m'(\vartheta))-a)}
\end{eqnarray}
if the random walk $(S_n)_{n\in\N_0}$ is non-arithmetic. If
$(S_n-an)_{n \in\mn_0}$ has span $\lambda_a > 0$, then,
analogously, as $n \to \infty$
\begin{eqnarray}    \label{eq:asymptotic_of_V_a_lattice}
V_a(e^{\lambda_a n})
& \sim &
\frac{\lambda_a e^{(\vartheta-1) \lambda_a n}}
{(1-e^{-\lambda_a(\vartheta-1)})(e^{a \vartheta}(-m^{\prime}(\vartheta)) -a)},  \\
\label{eq:comp_asymptotics_of_V_a_lattice} \int_{(0,e^{\lambda_a n}]} \!\! y \, V_a(dy)
& \sim & \frac{\lambda_a e^{\vartheta \lambda_a n}}
{(1-e^{-\lambda_a \vartheta})(e^{a \vartheta}(-m^{\prime}(\vartheta)) -a)}.
\end{eqnarray}
Further, if $\vartheta < 2$, then, in the non-arithmetic case, as $x \to \infty$,
\begin{equation}
\label{eq:tail_asymptotics_of_V_a}
\int_{(x,\,\infty)} \!\! y^{-1} \, V_a(dy) ~\sim~
\frac{x^{\vartheta-2}}{(2-\vartheta)(e^{a\vartheta}(-m'(\vartheta))-a)},
\end{equation}
whereas, in the arithmetic case, as $n \to \infty$,
\begin{equation}
\label{eq:tail_asymptotics_of_V_a_lattice}
\int_{[e^{\lambda_a n},\,\infty)} \!\! y^{-1} \, V_a(dy)    ~\sim~
\frac{\lambda_a e^{(\vartheta-2)\lambda_a n}}{(1-e^{(\vartheta-2)\lambda_a})(a^{a \vartheta}(-m'(\vartheta))-a)}.
\end{equation}
\end{Lemma}
\begin{proof}
Let $\varphi$ and $\varphi_a$ be the Laplace transforms of the
increment distributions of the associated random walk $(S_n)_{n
\in \N_0}$ defined by \eqref{eq:distribution_of_S_n} and of the
shifted random walk $(S_n - an)_{n \in \N_0}$, respectively. Our
purpose is to check that under the assumptions of the lemma
Theorem \ref{Theorem:Heyde} applies to the random walk
$(S_n-an)_{n\in\N_0}$ with $\psi=\varphi_a$ and $R= - \log \inf_{t
\geq 0} \varphi_a(t)$.

By definition, $\varphi_a(t) = e^{at} \varphi(t) = e^{at} m(1+t)$
which implies that
\begin{equation}    \label{eq:e^aphi_a(r-1)_leq_1}
e^{a} \varphi_a(r-1)  ~=~ e^{ar} m(r)  ~\leq~  1.
\end{equation}
Therefore, in the notation of Theorem \ref{Theorem:Heyde}
condition $a\leq R$ holds. In the case $a = -\log \inf_{t \geq
1}\, m^{1/t}(t)$, we have
$a=-\log \inf_{t \geq 0} \varphi_a(t) = R$ in view of
\eqref{eq:e^aphi_a(r-1)_leq_1}. With $\psi=\varphi_a$, $\gamma_0$
defined on p.\,\pageref{ga0} equals $\vartheta_0-1$. Therefore,
condition \eqref{eq:psi'(gamma_0)<0} reads
$\varphi^\prime_a(\vartheta_0-1)<0$ and is a consequence of
\eqref{eq:critical_derivative_negative}. In any case, Theorem
\ref{Theorem:Heyde} applies with $\gamma$ being the minimal $t> 0$
satisfying $\varphi_a(t) = e^{-a}$ that is, $\gamma = \vartheta -
1$, and yields
\begin{equation}    \label{eq:first_asymptotic_of_V_a}
V_a(x) ~\sim~ \frac{e^{-a}}{(\vartheta-1)
(-\varphi_a'(\vartheta-1))} \, x^{\vartheta-1} \qquad   (x \to
\infty),
\end{equation}
in case when $(S_n-an)_{n \in \N_0}$ is non-arithmetic, and
\begin{equation}    \label{eq:first_asymptotic_of_V_a_lattice}
V_a(e^{\lambda_a n})~\sim~ \frac{\lambda_a
e^{-a}}{(1-e^{-\lambda_a(\vartheta-1)})(-\varphi_a^\prime(\vartheta-1))}
\, e^{(\vartheta-1)\lambda_a n}
\qquad  (n \to \infty),
\end{equation}
in case when $(S_n-an)_{n\in\N_0}$ is arithmetic with span
$\lambda_a$.

Now first notice that \eqref{eq:first_asymptotic_of_V_a} proves
\eqref{eq:asymptotic_of_V_a} and that
\eqref{eq:first_asymptotic_of_V_a_lattice} implies
\eqref{eq:asymptotic_of_V_a_lattice}. Secondly, in the
non-arithmetic case, asymptotics
\eqref{eq:comp_asymptotics_of_V_a} and
\eqref{eq:tail_asymptotics_of_V_a} follow from
\eqref{eq:asymptotic_of_V_a} by integration by parts and
subsequent application of Propositions 1.5.8 and 1.5.10 in
\cite{BGT1989}, respectively. Finally, in the lattice case,
asymptotics \eqref{eq:comp_asymptotics_of_V_a_lattice} and
\eqref{eq:tail_asymptotics_of_V_a_lattice} follow by an elementary
analysis from \eqref{eq:asymptotic_of_V_a_lattice} and the
corresponding asymptotic for $V_a(e^{\lambda_a
n})-V_a(e^{\lambda_a (n-1)})$, which can be derived from
\eqref{eq:asymptotics_of_exp_renewal_measure2}. We omit the
details.
\end{proof}

\begin{Lemma}   \label{Lem:2_series}
Let $a>0$ be given. Assume that
\begin{equation*}
e^a m^{1/r}(r) ~\leq~ 1 \qquad  \text{for some } r > 1
\end{equation*}
and define $\vartheta$ to be the minimal $r > 1$ such that $e^a m^{1/r}(r) = 1$.
In case when $a = -\log \inf_{r \geq 1} m^{1/r}(r)$
(which implies $\vartheta=\vartheta_0$) assume further that
\begin{equation*}
-\frac{\log m(\vartheta_0)}{\vartheta_0} ~<~ -\frac{m^\prime(\vartheta_0)}{m(\vartheta_0)}.
\end{equation*}
Then $\E W_1^{\vartheta} < \infty$ implies
\begin{equation}    \label{eq:2_series}
\sum_{n \geq 0} \Prob(\widetilde{W}_{n+1} \not = W_{n+1}) ~<~ \infty
\quad   \text{and}  \quad
\E \sum_{n \geq 0} e^{an} R_n ~<~ \infty.
\end{equation}
If, moreover, $\vartheta < 2$, then $(M_n)_{n \in\mn_0}$ is an
$\mathcal{L}^2$-bounded martingale, where
\begin{equation*}
M_n ~:=~ \sum_{k=0}^n e^{ak}(\widetilde{W}_{k+1} - W_k + R_k),
\qquad  n \in \N_0.
\end{equation*}
\end{Lemma}

\begin{Rem} \label{Rem:R_n}
Note that the second infinite series in \eqref{eq:2_series} is
well-defined since all summands are non-negative. Indeed, for any
$n\in\N_0$, by the independence of $W_1(v)$ and $\F_n$ for
$|v|=n$,
\begin{eqnarray*}
R_n
& = &
\sum_{|v|=n} Y_v \E(1-W_1(v) \1_{\{e^{an} Y_v W_1(v) \leq 1\}}\,|\,\F_n)    \\
& = &
\sum_{|v|=n} Y_v \E(W_1(v) - W_1(v) \1_{\{e^{an} Y_v W_1(v) \leq 1\}}\,|\,\F_n) \\
& = & \sum_{|v|=n} Y_v \int_{(e^{-an} Y_v^{-1},\,\infty)} \!\! x
\, F(dx) ~\geq~ 0   \qquad   \text{a.s.}
\end{eqnarray*}
\end{Rem}

\begin{proof}[Proof of Lemma \ref{Lem:2_series}]
\begin{eqnarray*}
\sum_{n \geq 0} \Prob(\widetilde{W}_{n+1} \not = W_{n+1})
& \leq &
\sum_{n \geq 0} \E \sum_{|v|=n} \Prob(e^{an}Y_v W_1(v) > 1 \,|\,\F_n)   \\
& = &
\sum_{n \geq 0} \E \sum_{|v|=n} Y_v e^{S(v)} \int \1_{(e^{-an} e^{S(v)},\,\infty)}(x) \, F(dx)    \\
& = &
\sum_{n \geq 0} \E e^{S_n} \int \1_{(e^{S_n-an},\,\infty)}(x) \, F(dx)   \\
& = &
\int \sum_{n \geq 0} \E e^{S_n} \1_{\{e^{S_n - an} < x\}} \, F(dx)  \\
& = &
\int \sum_{n \geq 0} e^{an} \E e^{S_n - an} \1_{(0,x)}(e^{S_n - an}) \, F(dx)   \\
& = &
\int \int_{(0,x)} y \, V_a(dy) \, F(dx).
\end{eqnarray*}
Using \eqref{eq:comp_asymptotics_of_V_a} or \eqref{eq:comp_asymptotics_of_V_a_lattice}, respectively, yields
\begin{equation*}
\sum_{n \geq 0} \Prob(\widetilde{W}_{n+1} \not = W_{n+1})
~\leq~ \int O\left(x^{\vartheta}\right) \, F(dx)
~ < ~ \infty.
\end{equation*}
Concerning the second series in \eqref{eq:2_series}, we obtain by
using the calculations from Remark \ref{Rem:R_n} that
\begin{eqnarray*}
\E \sum_{n \geq 0} e^{an} R_n
& = &
\sum_{n \geq 0} e^{an} \E \sum_{|v|=n} Y_v \int_{(e^{-an}Y_v^{-1},\infty)} x \, F(dx)   \\
& = &
\sum_{n \geq 0} e^{an} \E \int_{(e^{S_n-an},\infty)} \!\! x \, F(dx)    \\
& = &
\int x \sum_{n \geq 0} e^{an} \Prob(S_n - an < \log x) \, F(dx) \\
& \leq &
\int x V_a(x) \, F(dx).
\end{eqnarray*}
In view of \eqref{eq:asymptotic_of_V_a} and \eqref{eq:asymptotic_of_V_a_lattice}, we conclude that
\begin{eqnarray*}
\E \sum_{n \geq 0} e^{an} R_n
& \leq &
\int x \, O\left(x^{\vartheta-1}\right) \, F(dx) \\
& = &
\int O \left(x^{\vartheta}\right) \, F(dx)
~<~ \infty.
\end{eqnarray*}

\noindent Now we turn to the final assertion of the lemma. Since
$R_n = \E(W_n - \widetilde{W}_{n+1} \,|\, \F_n)$, we have
$\widetilde{W}_{n+1} - W_n + R_n = \widetilde{W}_{n+1} -
\E(\widetilde{W}_{n+1} \,|\, \F_n)$ a.s. In particular, $(M_n)_{n
\geq 0}$ constitutes a martingale. It remains to prove that
$(M_n)_{n \in\mn_0}$ is $\mathcal{L}^2$-bounded. For this purpose,
note that
\begin{eqnarray*}
\E (e^{an}(\widetilde{W}_{n+1} - W_n + R_n))^2
& = &
e^{2an} \E \left(\widetilde{W}_{n+1} - W_n + R_n\right)^2   \\
& = &
e^{2an} \E \Var(\widetilde{W}_{n+1} \,|\, \F_n) \\
& \leq & e^{2an} \E \sum_{|v|=n} Y_v^2 \int_{(0,\,e^{-an}
Y_v^{-1}]} \!\! x^2 \, F(dx).
\end{eqnarray*}
Thus,
\begin{align*}
\sum_{n \geq 0} & \E(e^{an}(\widetilde{W}_{n+1} - W_n + R_n))^2 \\
& ~\leq ~
\sum_{n \geq 0} e^{2an} \E \sum_{|v|=n} Y_v^2 \int_{(0,\,e^{-an} Y_v^{-1}]} \!\! x^2 \, F(dx) \\
& ~=~
\sum_{n \geq 0} e^{an} \E \sum_{|v|=n} Y_v e^{an-S(v)} \int_{(0,\,e^{S(v)-an}]} \!\! x^2 \, F(dx) \\
& ~=~
\sum_{n \geq 0} e^{an} \E e^{-(S_n-an)} \int_{(0,\,e^{S_n-an}]} \!\! x^2 \, F(dx) \\
& ~=~
\int_{(0,\infty)} \!\! x^2 \sum_{n \geq 0} e^{an} \E e^{-(S_n-an)} \1_{\{e^{S_n-an} \geq x\}} \, F(dx) \\
& ~=~
\int_{(0,\infty)} \!\! x^2 \int_{[x,\infty)} y^{-1} \, V_a(dy) \, F(dx).
\end{align*}
Now use \eqref{eq:tail_asymptotics_of_V_a} or \eqref{eq:tail_asymptotics_of_V_a_lattice} to obtain the
finiteness of the last expression.
\end{proof}

The proof of the next result can be found in Remark 3.2 in
\cite{AIPR2009}.
\begin{Lemma}    \label{Prop:Exponential_convergence}
\eqref{eq:a.s._convergence_of_exponential_series} holds if and
only if
\begin{equation*}    \label{eq:a.s._convergence_of_exponential_series2}
\sum_{n \geq 0} e^{an} (W_{n+1}-W_n) \qquad  \text{converges a.s.}
\end{equation*}
\end{Lemma}

\begin{proof}[Proof of Theorem \ref{Theorem:exponential_convergence:sufficiency}]
Under the assumptions of the theorem, Lemma \ref{Lem:2_series}
implies that $(M_n)_{n\in\mn_0}$ is an $\mathcal{L}^2$-bounded
martingale, in particular,
\begin{equation*}
\sum_{n \geq 0} e^{an}(\widetilde{W}_{n+1} - W_n + R_n) \qquad   \text{converges a.s.}
\end{equation*}
This is equivalent to the a.s.\ convergence of the series $\sum_{n
\geq 0} e^{an}(\widetilde{W}_{n+1} - W_n)$ since $\sum_{n \geq 0}
e^{an} R_n < \infty$ a.s.\ in view of \eqref{eq:2_series}. Another
appeal to \eqref{eq:2_series} yields $\sum_{n \geq 0}
\Prob(\widetilde{W}_{n+1} \not = W_{n+1}) < \infty$ which directly
implies $\Prob(\widetilde{W}_{n+1} \not = W_{n+1} \text{ i.o.}) =
0$ by an application of the Borel-Cantelli Lemma. Hence,
\begin{equation*}
\sum_{n \geq 0} e^{an}(W_{n+1} - W_n)   \qquad   \text{converges a.s.,}
\end{equation*}
which is equivalent to
\eqref{eq:a.s._convergence_of_exponential_series} by Lemma
\ref{Prop:Exponential_convergence}.
\end{proof}

\bibliographystyle{abbrv}                   

\bibliography{Exponential_Rate}
\end{document}